\documentclass[reqno,12pt]{amsart}

\usepackage{amsmath,amssymb,mathtools,verbatim,slashed,bbm,upgreek,appendix}
\usepackage[backrefs]{amsrefs}
\usepackage[protrusion=true,babel=true]{microtype}
\usepackage[english]{babel}
\usepackage[margin=1in]{geometry}
\usepackage[onehalfspacing]{setspace}
\usepackage[pdfusetitle,pagebackref]{hyperref}
\numberwithin{equation}{section}							
\usepackage[nameinlink,noabbrev]{cleveref}
\expandafter\def\csname ver@etex.sty\endcsname{3000/12/31}

\usepackage{autonum}										

\newtheorem{theorem}{Theorem}[section]
\newtheorem{Mtheorem}{Main Theorem}
\newtheorem*{acknowledgment}{Acknowledgment}

\newtheorem{corollary}[theorem]{Corollary}

\newtheorem{lemma}[theorem]{Lemma}

\crefname{theorem}{Theorem}{Theorems}						
\creflabelformat{theorem}{#2{#1}#3}							
\crefname{Mtheorem}{Main Theorem}{Main Theorems}			
\creflabelformat{Mtheorem}{#2{#1}#3}						
\crefname{lemma}{Lemma}{Lemmata}							
\creflabelformat{lemma}{#2{#1}#3}							
\crefname{corollary}{Corollary}{Corollaries}				
\creflabelformat{corollary}{#2{#1}#3}						
\crefname{proposition}{Proposition}{Propositions}			
\creflabelformat{proposition}{#2{#1}#3}						
\crefname{ineq}{inequality}{inequalities}					
\creflabelformat{ineq}{#2{\upshape(#1)}#3}					
\crefname{cond}{condition}{conditions}						
\creflabelformat{cond}{#2{\upshape(#1)}#3}					
\crefname{hypoth}{Hypothesis}{Hypotheses}					
\creflabelformat{hypoth}{#2{#1}#3}							
\crefname{definition}{Definition}{Definitions}				
\creflabelformat{def}{#2{#1}#3}								
\crefname{appsec}{Appendix}{Appendices}

\allowdisplaybreaks[1]

\let\originalleft\left
\let\originalright\right
\renewcommand{\left}{\mathopen{}\mathclose\bgroup\originalleft}
\renewcommand{\right}{\aftergroup\egroup\originalright}

\makeatletter
\renewcommand*{\eqref}[1]{\hyperref[{#1}]{\textup{\tagform@{\ref*{#1}}}}}
\makeatother

\newcommand*{\eqdef}{\mathrel{\vcenter{\baselineskip0.5ex \lineskiplimit0pt\hbox{.}\hbox{.}}}=}

\def\id{\mathbbm{1}}
\def\cx{\mathbbm{C}}
\def\bL{\mathbbm{L}}
\def\rl{\mathbbm{R}}
\def\N{\mathbbm{N}}

\def\Z{\mathbbm{Z}}

\def\cB{\mathcal{B}}
\def\cC{\mathcal{C}}
\def\cD{\mathcal{D}}
\def\cE{\mathcal{E}}

\def\cG{\mathcal{G}}
\def\cH{\mathcal{H}}
\def\cI{\mathcal{I}}

\def\cL{\mathcal{L}}
\def\cM{\mathcal{M}}

\def\cT{\mathcal{T}}

\def\cV{\mathcal{V}}
\def\cX{\mathcal{X}}

\def\Ar{\mathrm{Area}}

\def\Im{\mathrm{Im}}

\def\index{\mathrm{index}}

\def\Re{\mathrm{Re}}

\def\Spec{\mathrm{Spec}}

\def\Sym{\mathrm{Sym}}

\def\rU{\mathrm{U}}

\def\vol{\mathrm{vol}}

\def\del{\partial}
\def\delbar{\overline{\partial}}
\def\rd{\mathrm{d}}

\def\Hess{\mathrm{Hess}}

\title[Nonminimal Ginzburg--Landau fields]{Nonminimal solutions to the Ginzburg--Landau equations on surfaces}
\date{\today}

\author{\'Akos Nagy}
\address[\'Akos Nagy]{University of California, Santa Barbara}
\urladdr{\href{https://akosnagy.com}{akosnagy.com}}
\email{\href{mailto:contact@akosnagy.com}{contact@akosnagy.com}}

\author{Gon\c{c}alo Oliveira}
\address[Gon\c{c}alo Oliveira]{IST Austria}
\urladdr{\href{https://sites.google.com/view/goncalo-oliveira-math-webpage/home}{sites.google.com/view/goncalo-oliveira-math-webpage/home}}
\email{\href{mailto:galato97@gmail.com}{galato97@gmail.com}}

\calclayout
\hypersetup{
	pdffitwindow	= true,
	unicode			= true,
	pdffitwindow	= true,
	pdftoolbar		= false,
	pdfmenubar		= false,
	pdfstartview	= {FitH},
	pdfauthor		= {\'Akos Nagy, Gon\c{c}alo Oliveira},
	colorlinks		= true,
	linkcolor		= black,
	citecolor		= black,
	filecolor		= black,
	urlcolor		= blue
}

\begin{document}

\begin{abstract}
	We prove the existence of novel, nonminimal and irreducible solutions to the (self-dual) Ginzburg--Landau equations on closed surfaces. To our knowledge these are the first such examples on nontrivial line bundles, that is, with nonzero total magnetic flux. Our method works with the 2-dimensional, critically coupled Ginzburg--Landau theory and uses the topology of the moduli space. The method is nonconstructive, but works for all values of the remaining coupling constant. We also prove the instability of these solutions.
\end{abstract}

\maketitle

\section{Introduction}
\label{sec:intro}

Ginzburg--Landau theory is one of the oldest gauge theoretic models of spontaneous symmetry breaking through the Higgs mechanism. The theory in dimension two can be summarized briefly as follows: Let $\left( \Sigma, g \right)$ be a closed, oriented, Riemannian surface whose volume form we denote by $\vol_g$. Fix a Hermitian line bundle $\left( \cL, h \right)$ over $\Sigma$ and a positive coupling constant $\tau \in \rl_+$. We follow the convention of physicists in that $h$ is complex linear in the second entry. For each smooth unitary connection $\nabla$ and smooth section $\upphi$ the \emph{Ginzburg--Landau energy} is given by
\begin{equation}
	\cE_\tau \left( \nabla, \upphi \right) = \int\limits_X \left( \left| F_\nabla \right|^2 + \left| \nabla \upphi \right|^2 + \tfrac{1}{4} \left( \tau - |\upphi|^2 \right)^2 \right) \ \vol_g. \label{eq:gle}
\end{equation}
The (classical) Ginzburg--Landau theory is the variational theory of the Ginzburg--Landau energy \eqref{eq:gle}. The corresponding Euler--Lagrange equations, called the \emph{Ginzburg--Landau equations}, are
\begin{subequations}
	\begin{align}
		\rd^* F_\nabla + i \: \Im \left( h \left( \upphi, \nabla \upphi \right) \right)	&= 0, \label{eq:gl1} \\
		\nabla^* \nabla \upphi - \tfrac{1}{2} \left( \tau - |\upphi|^2 \right) \upphi	&= 0. \label{eq:gl2}
	\end{align}
\end{subequations}
These are nonlinear, second order, elliptic partial differential equations which are invariant by the action of the group of automorphisms of $\left( \cL, h \right)$, also known as the gauge group. If a unitary connection $\nabla^0$ satisfies the abelian Yang--Mills equation (also known as the source-free Maxwell's equation)
\begin{equation}
	\rd^* F_{\nabla^0} = 0, \label{eq:normal}
\end{equation}
then the pair $\left( \nabla^0, 0 \right)$ solves the Ginzburg--Landau equations; such a pair is said to be a {\em normal phase solution}. Notice that \cref{eq:normal} is independent of $\tau$. As is common in abelian gauge theories, we call a pair $\left( \nabla, \upphi \right)$ {\em reducible} if $\upphi$ vanishes identically, and {\em irreducible} otherwise. A solution to the Ginzburg--Landau equations is reducible if and only if it is a normal phase solution.

\smallskip

On closed surfaces, the Ginzburg--Landau free energy is Palais--Smale by \cite{N17b}*{Main Theorem~2}, and thus has absolute minimizers which are automatically solutions to the Ginzburg--Landau equations. The minimizers, often called vortices, are well-understood; cf. \cites{JT80,B90,GP93} and more recently \cites{CERS17,N17b}. We remark that the Ginzburg--Landau energy \eqref{eq:gle} is often called \emph{critically coupled}, as there is a generalization of the energy \eqref{eq:gle} with a further coupling constant, (usually denoted by $\kappa$) which for is set to a special value ($\kappa = \tfrac{1}{2}$).

\smallskip

Much less is known about nonminimal solutions. In \cite{PS20}, Pigati and Stern constructed irreducible solutions on (topologically) trivial line bundles over closed Riemannian manifolds. As these solutions have positive energy, they cannot be the absolute minima of Ginzburg--Landau energy \eqref{eq:gle}. Finally, in our companion paper \cite{nagy_bifurcation_2022} we construct nonminimal and irreducible solutions on nontrivial line bundles over closed Riemannian manifolds with vanishing first Betti number.

\smallskip

In this paper, motivated by, and building on, works of \cites{Parker1992a,Parker1992b,Parker1993,N17b,CERS17}, we construct new, nonminimal and irreducible solutions to the Ginzburg--Landau equations. Furthermore, we prove the instability of these solutions, which extends the results of \cite{C20} to all closed surfaces, metrics, and degrees.

\medskip

\subsection*{Summary of main results}

Let $\left( \Sigma, g, \cL, h \right)$ be as above, with degree $d \eqdef c_1 (\cL) [\Sigma] \in \Z$. Without any loss of generality we can assume that $d \geqslant 0$. Let us define for the rest of the paper
\begin{equation}
	\tau_{\mathrm{Bradlow}} \eqdef \frac{4 \pi d}{\Ar \left( \Sigma, g \right)}.
\end{equation}

Our first main theorem shows the existence of such solutions in the situation above.

\begin{Mtheorem}\label{Mtheorem:nonminimal}
	Assume that $d \geqslant \mathrm{genus} \left( \Sigma \right)$. Then, for any positive integer $k \in \Z$, there is $\tau_k > \tau_{\mathrm{Bradlow}}$, such that for all generic $\tau > \tau_k$, the Ginzburg--Landau energy \eqref{eq:gle} has at least $k$ critical points which are neither vortices nor normal phase solutions.
\end{Mtheorem}

Our second main theorem completely classifies the local minima of the 2-dimensional critically coupled Ginzburg--Landau energy. This is an extension of the results of \cite{C20} to all closed surfaces, metrics, and degrees.

\begin{Mtheorem}
	\label{Mtheorem:instability}
	Under the hypothesis above, let $\left( \nabla, \upphi \right)$ be a stable critical point of the 2-dimensional, critically coupled Ginzburg--Landau energy. Then either:
	\begin{enumerate}

		\itemsep0em
		\parskip0em
		
		\item $\tau \leqslant \tau_{\mathrm{Bradlow}}$ and $\left( \nabla, \upphi \right) = \left( \nabla^0, 0 \right)$ is a normal phase solution.
		
		\item $\tau > \tau_{\mathrm{Bradlow}}$ and $\left( \nabla, \upphi \right)$ is a vortex field.
		
	\end{enumerate}	
	Equivalently, if $\left( \nabla, \upphi \right)$ is an irreducible critical point that is not a vortex field, then $\tau > \tau_{\mathrm{Bradlow}}$ and $\left( \nabla, \upphi \right)$ is unstable.
\end{Mtheorem}

\medskip

\subsection*{Organization of the paper} In \Cref{sec:GL_theory}, we give a brief introduction to the important geometric analytic aspects of the Ginzburg--Landau theory that are needed to prove our results. In \Cref{sec:existence} we study the topology of the configuration space and prove the existence of irreducible, nonvortex solutions in the 2-dimensional critically coupled Ginzburg--Landau theory (\Cref{Mtheorem:nonminimal}). Finally, in \Cref{sec:instability} we prove the instability of irreducible, nonvortex solutions in the 2-dimensional critically coupled Ginzburg--Landau theory (\Cref{Mtheorem:instability}).

\begin{acknowledgment}
	We are also grateful to Casey Kelleher for her help during the early stages of the project and for Da Rong Cheng for teaching us the proof of \Cref{theorem:index_bound}.

	The first author is thankful to Israel Michael Sigal and Steve Rayan for their insights on the topics of the paper.
	
	The second author was supported by the NOMIS Foundation, Funda\c{c}\~ao Serrapilheira 1812-27395, by CNPq grants 428959/2018-0 and 307475/2018-2, and FAPERJ through the program Jovem Cientista do Nosso Estado E-26/202.793/2019.
\end{acknowledgment}

\bigskip

\section{Ginzburg--Landau theory on surfaces}
\label{sec:GL_theory}

Let $\left( \Sigma, g \right)$ be a closed, oriented, Riemannian manifold of surface. Let us fix a connection $\nabla^0$ that satisfies \cref{eq:normal}. We define the Sobolev norms of $\cL$-valued forms via the Levi-Civita connection of $\left( \Sigma, g \right)$ and the connection $\nabla^0$. Note that the induced topologies are the same for any choice of $\nabla^0$ and since the moduli of normal phase solutions is compact (modulo gauge), and if a Coulomb-type gauge fixing condition is chosen (with respect to a reference connection), then the family of norms are in fact uniformly equivalent, that is, for all $k$ and $p$, there exists a number $C_{k, p} \geqslant 1$, such that for any two Sobolev $L_k^p$-norms, $\| \cdot \|_{L_k^p}$ and $\| \cdot \|_{L_k^p}^\prime$, given by two connections satisfying \cref{eq:normal} and the Coulomb condition, we have
\begin{equation}
	\frac{1}{C_{k, p}} \| \cdot \|_{k, p}^\prime \leqslant \| \cdot \|_{k, p} \leqslant C_{k, p} \| \cdot \|_{k, p}^\prime.
\end{equation}

Let $\Omega^k$ be the $L_1^2$-completion of the space of smooth $k$-forms. Let $\cC_\cL$ be the $L_1^2$-completion of smooth unitary connections on $\cL$, which is an affine space over $i \Omega^1$, and thus the tangent bundle, $T \cC_\cL$, is canonically isomorphic to $\cC_\cL \times i \Omega^1$. In particular, $\cC_\cL$ is a Hilbert manifold. Let $\Omega_{\rd^*}^1 \subset \Omega^1$ be the kernel of $\rd^*$ and $\cC_{\cL, \rd^*} \eqdef \nabla^0 + \Omega_{\rd^*}^1$. Finally, let $\Omega_\cL^k$ be the $L_1^2$-completion of the space of smooth $\cL$-valued $k$-form. Similarly, we define $\Omega_\cL^{p, q}$.

With these definitions, the Ginzburg--Landau energy \eqref{eq:gle} is an real-analytic function on $\cX \eqdef \cC_{\cL, \rd^*} \times \Omega_\cL^0$. In other words, we are working with the Coulomb gauge fixing. Without gauge fixing, \cite{JT80}*{page~204,~Theorem~2.4} shows that every critical point is gauge equivalent to a smooth one, which in turn is a solution to \cref{eq:gl1,eq:gl2}. Equivalently, with the Coulomb gauge fixing condition, we can use elliptic regularity to show the same. Thus gauge equivalence classes of critical points of the Ginzburg--Landau energy \eqref{eq:gle} are in one-to-one correspondence with gauge equivalence classes of smooth solutions of the \cref{eq:gl1,eq:gl2}.

Note that harmonic gauge transformations are not eliminated by the Coulomb gauge fixing condition and such transformations do not act freely on configurations of the form $\left( \nabla, 0 \right)$. To resolve this problem, we fix $x \in \Sigma$ and then we can identify harmonic gauge transformations which are the identity at $x$ with $H^1 \left( \Sigma; \Z \right)$. We then define
\begin{equation}
	\cB \eqdef \cX / H^1 \left( \Sigma; \Z \right). \label{eq:config_space}
\end{equation}
Since the Ginzburg--Landau energy \eqref{eq:gle} is gauge invariant, it descends to a functional on $\cB$, which we denote the same way. Now let
\begin{equation}
	\cM \eqdef \left\{ \ (\nabla^0 + A, 0) \ | \ \rd A = 0 \ \right\} / H^1 \left( \Sigma; \Z \right) \subset \cB, \label{eq:harmonic_connections}
\end{equation}
the moduli space of normal phase solutions. One can easily see that
\begin{equation} 
	\cM \cong H^1 \left( \Sigma; \rl \right) / H^1 \left( \Sigma; \Z \right),
\end{equation}
Thus $\cM$ is a torus of (real) dimension $2 \cdot \mathrm{genus} \left( \Sigma \right)$.

\bigskip

\section{Nonminimal solutions through the topology of the configuration space}
\label{sec:existence}

In this section we prove \Cref{Mtheorem:nonminimal}. Without any loss of generality, assume that $d \eqdef \deg (\cL) \geqslant 0$.

On the one hand, when $\tau \leqslant \tau_{\mathrm{Bradlow}} = \tfrac{4 \pi d}{\Ar \left( \Sigma, g \right)}$, then by \cite{N17b}*{Main Theorem~2}, the only critical points are the normal phase solutions, thus we also assume that $\tau > \tau_{\mathrm{Bradlow}}$. On the other hand, when $\tau > \tau_{\mathrm{Bradlow}}$, then Bradlow showed in \cite{B90} that the moduli space of the absolute minimizers of the critically coupled Ginzburg--Landau energy (called vortices) is diffeomorphic to $\Sym^d \left( \Sigma \right)$, which is a smooth, complex $d$-dimensional K\"ahler manifold.

\smallskip

\subsection{The Hessian at a normal phase solution}
\label{ss:Hessian_Mormal_Phase_solutions}

Let $\left( \nabla^0, 0 \right)$ be a normal phase solution and define $\Delta_0 \eqdef \left( \nabla^0 \right)^* \nabla^0$ on $\Omega_\cL^0$. The Hessian of the Ginzburg--Landau energy at $\left( \nabla^0, 0 \right)$ is
\begin{equation}
	\Hess \left( \cE_\tau \right)_{\left( \nabla^0, 0 \right)} (a, \psi) = 2 \langle(a, \psi) | Q (a, \psi) \rangle_{L^2},
\end{equation}
where
\begin{equation}
	Q (a, \psi) \eqdef \begin{pmatrix} Q_1 (a, \psi) \\ Q_2 (a, \psi) \end{pmatrix} = \begin{pmatrix} \rd^* \rd a \\ \Delta_0 \psi - \tfrac{\tau}{2} \psi \end{pmatrix}.
\end{equation}
Considered as a densely defined, self-adjoint operator on the $L^2$-completion of $i \Omega_{\rd^*}^1 \times \Omega_\cL^0$, the spectrum of the operator $Q$ is bounded from below, contains only countably many eigenvalues (without accumulation points), and each eigenspace is finite dimensional. With that in mind, we define the (Morse-)index of the Hessian to be the dimension of the negative eigenspace of $Q$, that is
\begin{equation}
	\mbox{The (Morse-)index of } \Hess \left( \cE_\tau \right)_{\left( \nabla^0, 0 \right)} = \sum\limits_{\lambda \in (- \infty, 0)} \dim_\rl \left( \ker \left( Q - \lambda \id \right) \right).
\end{equation}
We now bound this index from below by looking for eigensections of the operator $Q$ with negative eigenvalues. This leads to the following result.

\begin{lemma}\label{lemma:Hessian_Maxwell}
	For any $N \in \N$, there is a $\tau_N$ such that for all $\tau > \tau_N$ the Hessian of $\cE_\tau$ at $\left( \nabla^0, 0 \right)$ has index at least $N$.
\end{lemma}

\begin{proof}
	We can construct eigensections of $Q$ with a negative eigenvalues of the form $(0, \psi)$ with $\psi$ being an eigensection of $\Delta_0 - \tfrac{\tau}{2} \id$. Let $\mu_1 \leqslant \mu_2 \leqslant \ldots$ be the eigenvalues of $\Delta_0$. Then the eigenvalues of $\Delta_0 - \tfrac{\tau}{2} \id$ are \
	\begin{equation}
		\mu_1 - \tfrac{\tau}{2} \leqslant \mu_2 - \tfrac{\tau}{2} \leqslant \cdots \leqslant \mu_N - \tfrac{\tau}{2} \leqslant \ldots.
	\end{equation}
	Thus setting $\tau_N > 0$ so that
	\begin{equation}
		N \leqslant \dim_\rl \left( \bigoplus\limits_{2 \mu < \tau} \ker \left( \Delta_0 - \tfrac{\tau}{2} \id \right) \right),
	\end{equation}
	concludes the proof. In particular, $\tau_N \eqdef 2 \mu_N$ works.
\end{proof}

\smallskip

In our next lemma we study the behavior of $\cE_\tau$ on $\cM \subseteq \cB$; cf \cref{eq:config_space,eq:harmonic_connections}. Note that now $\cM$ is a torus of dimension $b \eqdef b_1 (\Sigma) = 2 \mathrm{genus} (\Sigma)$.

\begin{lemma}\label{lem:Almost_Morse_Bott}
	There is a residual subset $S \subseteq \rl_+$, such that if $\tau \in S$, then there is an open and dense subset $\cM_\tau \subseteq \cM$ such that for each $[(\nabla, 0)] \in \cM_\tau$ there is an open subset $U \subseteq \cB$ such that $[(\nabla, 0)] \in U$ and there are coordinates (for some $c \in \N_+$)
	\begin{equation}
		(x, u, v) = (x_1, \ldots, x_b, u^-_1, \ldots u^-_c, u^+) : U \rightarrow \rl^b \times \rl^c \times l^2 (\N),
	\end{equation} 
	so that $\cM \cap U = (u^+)^{-1} (0) \cap (u^-)^{-1} (0)$ and  
	\begin{equation}
		\sum\limits_{i = 1}^c u_i^- \frac{\del \cE_\tau}{\del u_i^-}(x,u^-,u^+) < 0, \quad \& \quad \sum\limits_{i = 1}^\infty u_i^+ \frac{\del \cE_\tau}{\del u_i^+}(x,u^-,u^+) > 0
	\end{equation}
	for $u^- \neq 0$ and $u^+ \neq 0$ respectively.
\end{lemma}

\begin{proof}
	Throughout the proof we use that $T_{[(\nabla^0, 0)]} \cB \cong \Omega_{\rd^*}^1 \oplus \Omega_\cL^0$ and the isomorphism is canonical. 
	
	We start by showing that the set of points of $\cM$ at which $\cE_\tau$ is Morse--Bott, in the sense of \cite{Feehan2020}*{Definition~1.5}, is an open dense set. By definition these are the points at which the restriction of $\mathrm{Hess} \left( \cE_\tau \right)$ orthogonally decomposes the tangent bundle as follows:
	\begin{equation}
		T \cB \cong T \cM \oplus V_+ \oplus V_-, \label{eq:TB_decomp}
	\end{equation}
	with $V_\pm$ denoting the span of the positive and negative eigenspaces of $Q = (Q_1,Q_2)$. 
	
	The operator $Q_1 = \rd^* \rd$ is self-adjoint, elliptic (thus, Fredholm), it is nonnegative on $\Omega_{\rd^*}^1$ and its kernel is exactly the space harmonic 1-forms, which is also canonically the same as $T_{[(\nabla^0, 0)]} \cM$. Similarly, $Q_2 = \nabla^* \nabla - \tfrac{\tau}{2} \id_{\Omega_\cL^0}$ is self-adjoint and Fredholm, so it is enough to show that generically (for a residual set of $\tau$'s) zero is not in its spectrum.
	
	For all $\tau \in \rl$, let
	\begin{equation}
		\cM_\tau \eqdef \left\{ \ [\nabla] \in \cM \ \middle| \ 0 \notin \Spec \left( \nabla^* \nabla - \tfrac{\tau}{2} \id_{\Omega_\cL^0} \right) \ \right\} \subseteq \cM.
	\end{equation}
	Thus $\cM_\tau$ is exactly the set of Morse--Bott points.
	
	Recall that if $D$ is a Fredholm operator, then there is $\epsilon = \epsilon (D) > 0$, such that if $D^\prime$ is $\epsilon$-close to $D$ (in operator norm), then $D^\prime$ is also Fredholm and
	\begin{equation}
		\dim \left( \ker \left( D^\prime \right) \right) \leqslant \dim \left( \ker \left( D \right) \right). \label[ineq]{ineq:dim_drop}
	\end{equation}
	Thus $\cM_\tau$ is always open.
	
	Let
	\begin{equation}
		S \eqdef \left\{ \ \tau \in \rl \ \middle| \ \cM_\tau \mbox{ is dense} \ \right\}.
	\end{equation}
	We need to show that $S \subseteq \rl$ is residual. Pick a countable dense set $\cD \subseteq \cM$. For each $[\nabla] \in \cD$, let
	\begin{equation}
		S_{[\nabla]} \eqdef \rl - \Spec \left( 2 \nabla^* \nabla \right).
	\end{equation}
	Thus $\tau \in S_{[\nabla]}$ exactly when $0 \notin \Spec \left( \nabla^* \nabla - \tfrac{\tau}{2} \id_{\Omega_\cL^0} \right)$. Furthermore, $S_{[\nabla]}$ is open and dense in $\rl$. Let
	\begin{equation}
		S^\prime \eqdef \bigcap\limits_{[\nabla] \in \cD} S_{[\nabla]}.
	\end{equation}
	Then for all $\tau \in S^\prime$, we have that $\cD \subseteq \cM_\tau$ and thus $\cM_\tau$ is dense. Hence $S^\prime \subseteq S$. Then, by the Baire category theorem, $S \subseteq \rl$ is residual, as it contains a countable intersection of open dense sets.
	
	The orthogonality in \cref{eq:TB_decomp} follows from the fact that $Q$ is self-adjoint. The rest then follows from the Morse--Bott Lemma in \cite{Feehan2020}*{Theorem 4}.
\end{proof}

\medskip

\subsection{Perturbing the Ginzburg--Landau energy}

Recall that we assume $d = \deg (\cL) \geqslant 0$ and $\tau > \tau_{\mathrm{Bradlow}} = \tfrac{4 \pi d}{\mathrm{Area}\left( \Sigma \right)}$. Then we have two special submanifolds of $\cB$: the moduli space of normal phase solutions, $\cM$, and the framed moduli space of vortices, which we call $\cV$. Recall from \Cref{sec:GL_theory} that $\cM$ is always isomorphic to the Jacobian of $\Sigma$, which is a torus of dimension $2 \: \mathrm{genus} \left( \Sigma \right)$. In \Cref{app:MB_of_V}, we prove that $\cV \subset \cB$ is a smooth and closed submanifold of $\cB$ and $\cE_\tau$ is a Morse--Bott function around $\cV$. As all elements of $\cV$ are irreducible, the remaining gauge action of $\rU (1)$ acts freely on $\cV$. Thus $\cV$ is a principal $\rU (1)$-bundle over the moduli space of vortices, which in turn is canonically isomorphic to $\Sym^d \left( \Sigma \right)$; cf. \cite{B90}.

In this section we perturb the Ginzburg--Landau energy \eqref{eq:gle} so that it becomes a Morse function near $\cM$ and $\cV$, but all other critical points are unchanged.

Let us make a few definitions first. Pick $\delta  > 0$ small enough so that the $\delta$-neighborhoods of $\cM$ and $\cV$ in $\cB$, which we call $U_\cM$ and $U_\cV$, respectively, are tubular neighborhoods. By \cite{N17b}*{Main~Theorem~4}, we can assume that $\cE_\tau$ has no critical points in $\left( U_\cM \cap U_\cV \right) - \left( \cM \cap \cV \right)$. Identifying $U_\cM$ and $U_\cV$ with tubular neighbourhoods of the zero section in the normal bundles of $\cM$ and $\cV$ respectively, we shall denote by $\pi_\cM: U_\cM \rightarrow \cM$ and $\pi_\cV: U_\cV \rightarrow \cV$ the respective projections. Then, we choose bump functions $\chi_\cM$ and $\chi_\cV$ respectively supported in $U_\cM$ and $U_\cV$, and which take value 1 on the $\tfrac{\delta}{2}$-neighborhoods, $\widetilde{U}_\cM \subset U_\cM$ and $\widetilde{U}_\cV \subset U_\cV$, of $\cM$ and $\cV$, respectively.

Finally, we shall pick perfect Morse functions $f_\cM : \cM \rightarrow \rl$ and $f_\cV : \cV \rightarrow \rl$, so that the critical points of $f_\cM$ do not coincide with the non-Mose--Bott points of $\cE_\tau$, i.e. they are contained in the subset $\cM_\tau \subset \cM$ given by Lemma \ref{lem:Almost_Morse_Bott}. 

Now we use functions with these properties to perturb the Ginzburg--Landau energy $\cE_\tau$ on $\cB$ as follows. Let $\epsilon >0$ be a small positive number to be fixed at a later stage, the modified Ginzburg--Landau energy is
\begin{equation}
	\cE_\tau^\epsilon \eqdef \cE_\tau + \epsilon  \left( \chi_\cM (f_\cM \circ \pi_\cM ) + \chi_\cV ( f_\cV \circ \pi_\cV ) \right),
\end{equation}
Notice that $\cE_\tau^\epsilon$ coincides with $\cE_\tau$ outside $U_\cM \cup U_\cV$. 

\smallskip

The main result of this section considers the critical points of $\cE_\tau^\epsilon$. It is the following.

\begin{theorem}\label{theorem:perturbed_E_tau}
	There are functions $f_\cM$, $f_\cV$, and positive constants $\delta$, $\epsilon$, all as in the discussion above, such that the following statement holds. For generic $\tau > \tau_{\mathrm{Bradlow}}$, the critical points of $\cE_\tau^\epsilon$ coincide with those of $\cE_\tau$ away from $\cM \cup \cV$, and with those of $f_\cM$ and $f_\cV$ in $\cM$ and $\cV$ respectively.
\end{theorem}

\begin{proof}
	The claim is local, thus we can determine the location of the critical points by separating the analysis into three different regions: $\cB - (U_\cM \cup U_\cV)$, $\widetilde{U}_\cM \cup \widetilde{U}_\cV$, and $(U_\cM - \widetilde{U}_\cM) \cup (U_\cV - \widetilde{U}_\cV)$. Throughout the proof we shall be assuming that we have chosen a generic $\tau$ for which Lemma \ref{lem:Almost_Morse_Bott} applies.
	
	\subsubsection*{First region, i.e. $\cB - (U_\cM \cup U_\cV)$:} In this region $\chi_\cM=0=\chi_\cV$ and so the perturbed functional coincides with the unperturbed one, i.e. $\cE_\tau = \cE_\tau^\epsilon$, and so their critical points coincide there.
	
	\subsubsection*{Second region, i.e. $\widetilde{U}_\cM \cup \widetilde{U}_\cV$:} We now show that we can arrange $f_\cM$ and $\widetilde{U}_\cM$ (i.e. $\delta$) so that the only critical points of $\cE_\tau^\epsilon$ in $\widetilde{U}_\cM$ lie in $\cM$ and coincide with those of $f_\cM$. We have to distinguish between two situations, when we are in $\cM_\tau$ which consists of the Morse--Bott points and when we are away from these, i.e. in $\cM - \cM_\tau$. It is important to remark that what we do in this region works for all positive $\epsilon$.
	
	The same argument also shows that the only critical points of $\cE_\tau^\epsilon$ in $ \widetilde{U}_\cV$ actually coincide with those of $f_\cV$ and lie in $\cV$. In fact, the argument in that case is easier as there are no non-Morse--Bott points as shown in the Appendix \ref{app:MB_of_V}.
	
	We start withe the situation near a non-Morse--Bott point, $p \in \cM - \cM_\tau$, of $\cE_\tau$. First fix an open neighborhood $B_p \subseteq \cM$ of $p \in \cM$ of $\cE_\tau$, and consider $\widetilde{U}_\cM \cap \pi_\cM^{- 1} (B_p)$. If $\dim (\cM) = 2 \: \mathrm{genus} (\Sigma) = 0$ then $\cM$ is a point and we can choose $f_\cM=0$. Hence, we consider the case when $\dim (\cM) = 2 \: \mathrm{genus} (\Sigma)$ is positive and therefore at least two. Since $\mathrm{codim} \left( \ker \left( D \cE_\tau \right) \right) \leqslant 1$ and $\dim (T \cM) > 1$, we can choose a nonzero vector-field $X \in \ker \left( D \cE_\tau \right) \cap \pi_\cM^* T\cM$ on $\widetilde{U}_\cM \cap \pi_\cM^{- 1} (B_p)$. By possibly shrinking $B_p$ we may assume that and $D f_\cM((\pi_\cM)_\ast X)$ never vanishes on $B_p$. Indeed, if this was not possible, then for any such vector fields $X$ there would exist a sequence of points $\lbrace q_n \rbrace_{n \in \mathbb{N}} \subset \cM$ with $q_n \to p$ and $\left( D f_\cM ((\pi_\cM)_* X) \right)_{q_n}=0$. Given that this must be true for all vector fields $X$ as above, which generate $T_p\cM$ upon restriction to $p$, this would imply that $p$ is a critical point of $f_\cM$ which contradicts our choice of $f_\cM$.
	
	Summarizing, we have such a vector field for which $D \widetilde{f}_\cM(X)$ never vanishes in $\widetilde{U}_\cM \cap \pi_\cM^{- 1} (B_p)$. Then, by construction we have
	\begin{equation}
		D \cE_\tau^\epsilon (X) = D \cE_\tau (X) + \epsilon D\widetilde{f}_\cM(X) = \epsilon D \widetilde{f}_\cM(X) \neq 0,
	\end{equation}
	independently of $\epsilon>0$. By \Cref{lem:Almost_Morse_Bott}, the set of non-Morse--Bott is a closed subset of $\cM$ and thus compact. Hence, we can construct an open subset $B$ of $\cM - \cM^\prime$ in this manner, so that $\cE_\tau^\epsilon$ has no critical point in the region $\pi_\cM^{- 1} (B) \cap \widetilde{U}_\cM$. 
	
	Now, we investigate the existence of critical points of $\cE_\tau^\epsilon$ in the $\widetilde{U}_\cM$-complement of this region, that is, in $\widetilde{U}_\cM - (\widetilde{U}_\cM \cap \pi_\cM^{- 1} (B))$. At any critical point in this region, we must have
	\begin{equation}\label{eq:Intermediate computation critical pint}
		D \cE_\tau = - \epsilon  \chi_\cM  D ( \pi_\cM^* f_\cM ).
	\end{equation}
	Then, for the generic $\tau > \tau_{\mathrm{Bradlow}}$, we can use the coordinates $(x,u^+,u^-)$ from \Cref{lem:Almost_Morse_Bott}, we find that if $(u^+, u^-) \neq (0, 0)$, either $u_i^+\del_{u_i^+} \cE_\tau \neq 0$ or $u_i^-\del_{u_i^-} \cE_\tau \neq 0$. Putting this together with \ref{eq:Intermediate computation critical pint}, we find that at any critical point having $(u^+,u^-) \neq (0,0)$, we would either have $\del_{u^+_i} ( \pi_\cM^* f_\cM ) \neq 0$ or $\del_{u^-_i} ( \pi_\cM^* f_\cM )  \neq 0$, which is impossible as $\pi_\cM^* f_\cM$ is pulled back from $\cM$ and so only depends on the $x$ coordinates. Hence, at any critical point of $\cE_\tau^\epsilon$ in $\widetilde{U}_\cM$ we must have both $u^+=0=u^-$ and so it lies in $\cM$. Thus, since $D \cE_\tau = 0$ along $\cM$, we must also have $D f_\cM=0$ as we wanted to show. Notice in particular that this argument works independently of the value of $\epsilon$, as long as it is nonzero.
	
	\smallskip
	
	\subsubsection*{Third region, i.e.  $(U_\cM - \widetilde{U}_\cM) \cup (U_\cV - \widetilde{U}_\cV)$:} We argue by contradiction and assume that there is a critical point of $\cE_\tau^\epsilon$ in $(U_\cM - \widetilde{U}_\cM) \cup (U_\cV - \widetilde{U}_\cV)$. At that point we have
	\begin{equation}\label{eq:Crit_Points_Intermediate}
		D \cE_\tau = - \epsilon \ \left(  D ( \chi_\cM \pi_\cM^* f_\cM ) + D ( \chi_\cV \pi_\cV^* f_\cV ) \right).
	\end{equation}
	By the hypotheses of the lemma, we can assume that $|D \cE_\tau|$ is uniformly bounded below by a positive number on $(U_\cM - \widetilde{U}_\cM) \cup (U_\cV - \widetilde{U}_\cV)$, and thus we can choose $\epsilon > 0$ so that
	\begin{equation}
		\epsilon < \tfrac{1}{2} \min \left\{ \tfrac{\inf_{ U_\cV - \widetilde{U}_\cV } |D \cE_\tau|}{\sup_{U_\cV - \widetilde{U}_\cV} | D ( \chi_\cV \pi_\cV^* f_\cV )|}, \tfrac{\inf_{ U_\cM - \widetilde{U}_\cM } |D \cE_\tau|}{\sup_{U_\cM - \widetilde{U}_\cM} |D ( \chi_\cM \pi_\cM^* f_\cM )|} \right\}. \label{eq:choose_epsilon}
	\end{equation}
	However \cref{eq:choose_epsilon,eq:Crit_Points_Intermediate} cannot hold simultaneously, and thus there can be no critical points of $\cE_\tau^\epsilon$ in $(U_\cM - \widetilde{U}_\cM) \cup (U_\cV - \widetilde{U}_\cV)$.
\end{proof}

\medskip

\subsection{The topology of the configuration space}

In order to be able to use Morse Theory using $\cE_\tau^\epsilon$, in this section, we study the (weak) homotopy type of $\cB$.

\smallskip

Recall that $\cB$ is defined in \cref{eq:config_space} as the quotient of $\cX$ by $\cG_0$, which consists of the gauge transformations that are the identity at an initially chosen base point $x_0 \in \Sigma$. In particular, $H^1 (\Sigma; \Z) \hookrightarrow \cG_0$ as harmonic, $\rU (1)$-valued functions (that vanish at $x_0$). Moreover $\cG \cong \cG_0 \times \rU (1)$. Fix another base point $[\ast] \in \mathbb{CP}^\infty$ and let $\mathrm{Map}^0 (\Sigma , \mathbb{CP}^\infty)_\cL$ be the space of base point preserving maps that pullback the generator of $H^2 (\mathbb{CP}^\infty, \Z)$ to $c_1 (\cL) \in H^2(\Sigma; \Z)$ equipped with the compact open topology. The following result from \cite{Donaldson1990}*{Proposition~5.1.4} computes the weak rational homotopy type of $\cB$.

\begin{lemma}
	There is a weak rational homotopy equivalence $\cB \cong_{\mathbb{Q}} \mathrm{Map}^0 (\Sigma, \mathbb{CP}^\infty)_\cL$.
\end{lemma}

Thus we have the following result as well.

\begin{corollary}
\label{corollary:B_homotopy}
	There is a rational weak homotopy equivalence $\cB \cong_\mathbb{Q} K (H^1(\Sigma , \Z) , 1)$.
\end{corollary}

\begin{proof}
	The strategy we follow uses the long exact sequence of homotopy groups induced by the fibration
	\begin{equation}\label{eq:Fibration_B}
		\begin{matrix}
			\mathrm{Map}^0(\Sigma,\mathbb{CP}^\infty) & \rightarrow & \mathrm{Map}(\Sigma,\mathbb{CP}^\infty) \\
			& & \downarrow \\
			& & \mathbb{CP}^\infty,
		\end{matrix}
	\end{equation}
	and a theorem of Ren\'e Thom to compute $\mathrm{Map}(\Sigma,\mathbb{CP}^\infty)$. This says that for $m \in \mathbb{N}$
	\begin{equation}\label{eq:Thom}
		\mathrm{Map}(\Sigma, K (\Z, m)) ) \cong_{\mathbb{Q}} \prod_{j=0}^l K (H^j(\Sigma; \Z) , m-j ).
	\end{equation}
	As $\mathbb{CP}^\infty = K (\Z, 2) $, when applied to the case at hand we find
	\begin{equation}\label{eq:Map Sigma CP}
		\mathrm{Map}(\Sigma, \mathbb{CP}^\infty ) \cong_{\mathbb{Q}} K (H^0 (\Sigma; \Z),2) \times K (H^1(\Sigma; \Z),1) \times K (H^2(\Sigma; \Z),0).
	\end{equation}
	Then the long exact sequence in rational homotopy groups induced by the fibration \eqref{eq:Fibration_B} gives $\pi_k^{\mathbb{Q}}(\mathrm{Map}^0(\Sigma,\mathbb{CP}^\infty )) = \pi_k^{\mathbb{Q}} (\mathrm{Map}(\Sigma,\mathbb{CP}^\infty ))$ for $k \neq 1,2$. On the other hand, for these values of $k$ we find instead that
	\begin{align}
		0 \rightarrow \pi_2^{\mathbb{Q}}(\mathrm{Map}^0) \xrightarrow{i} \pi_2^{\mathbb{Q}}( \mathrm{Map} ) \xrightarrow{\mathrm{ev}} \pi_2^{\mathbb{Q}} (\mathbb{CP}^\infty) \rightarrow \pi_1^{\mathbb{Q}} ( \mathrm{Map}^0 ) \xrightarrow{j} \pi_1^{\mathbb{Q}}( \mathrm{Map} ) \rightarrow 0 ,
	\end{align}
	where $\mathrm{Map} \eqdef \mathrm{Map}(\Sigma,\mathbb{CP}^\infty)$ and $\mathrm{Map}^0 \eqdef \mathrm{Map}^0(\Sigma,\mathbb{CP}^\infty)$. From this we now prove that $\rm{ev}$ is an isomorphism. Indeed, if the map $g : S^2 \rightarrow \mathbb{CP}^\infty$ generates $\pi_2(\mathbb{CP}^\infty )$, we can consider the map $\widetilde{g}: S^2 \rightarrow \mathrm{Map}(\Sigma,\mathbb{CP}^\infty)$ which for $s \in S^2$ yields the constant map 
	\begin{equation}
		\widetilde{g}_s : \Sigma \rightarrow \mathbb{CP}^\infty,
	\end{equation}
	with $\widetilde{g}_s(x)=g(s)$ for all $x \in \Sigma$. Then $\mathrm{ev} (\widetilde{g}) \eqdef \mathrm{ev} \circ \widetilde{g} = g$ and so
	\begin{equation}
		\mathrm{ev} : \pi_2^{\mathbb{Q}} \left( \mathrm{Map} \right) \rightarrow \pi_2^{\mathbb{Q}} \left( \mathbb{CP}^\infty \right) \cong \mathbb{Q},
	\end{equation}
	is surjective. Given that $\pi_n(X \times Y) \cong \pi_n(X) \times \pi_n(Y)$ for any topological spaces $X,Y$, we find from \cref{eq:Map Sigma CP} that $\pi_2^{\mathbb{Q}} ( \mathrm{Map} ) \cong \mathbb{Q}$. Hence,
	\begin{equation}
		\pi_2^{\mathbb{Q}} \left( \mathrm{Map} \right) \cong \mathbb{Q} \cong \pi_2^{\mathbb{Q}} \left( \mathbb{CP}^\infty \right),
	\end{equation}
	$\mathrm{ev}$ is therefore also injective and so
	\begin{equation}
		\pi_2^{\mathbb{Q}} \left( \mathrm{Map}^0 \right) = 0.
	\end{equation}
	Finally, we conclude that $\pi_1^{\mathbb{Q}}(\mathrm{Map}^0) \cong \pi_1^{\mathbb{Q}} ( \mathrm{Map} )$ which together with the above gives 
	\begin{equation}
		\mathrm{Map}^0 \cong_\mathbb{Q} K \left( H^1 \left( \Sigma; \Z \right), 1 \right) \times K \left( H^2 \left( \Sigma; \Z \right), 0 \right).
	\end{equation}
	Notice that the implied fact that $\pi_0 \left( \mathrm{Map}^0 \left( \Sigma, \mathbb{CP}^\infty \right) \right) = H^2 \left( \Sigma; \Z \right)$, constitutes the statement that $\cL$ is topologically determined by $c_1(\cL)$ and so
	\begin{equation}\label{eq:Homotopy_Type_B}
		\cB \cong_\mathbb{Q} \mathrm{Map}^0(\Sigma, \mathbb{CP}^\infty )_\cL \cong_\mathbb{Q} K (H^1(\Sigma; \Z) , 1),
	\end{equation}
	as claimed in the statement.
\end{proof}

\medskip

\subsection{New Ginzburg--Landau fields}

In this subsection we complete the argument showing the existence of other Ginzburg--Landau fields than those in $\cM$ or $\cV$. 

\begin{proof}[Proof of \Cref{Mtheorem:nonminimal}:]
	Arguing by contradiction, suppose that the only critical points of $\cE_\tau$ are those in $\cM \cup \cV$. In particular, \Cref{theorem:perturbed_E_tau} applies. Then we have that $\cE_\tau (\cV) < \cE_\tau (\cM)$ and there is a perturbation $\cE_\tau^\epsilon$ as above with
	\begin{equation}
		\sup_\cV \cE_\tau^\epsilon < \inf_\cM \cE_\tau^\epsilon,
	\end{equation}
	and whose critical points coincide with those of the perturbations $f_\cM$ and $f_\cV$ on $\cM$ and $\cV$, respectively. Furthermore, by construction the function $\cE_\tau^\epsilon$ is Morse and we can perturb the metric so that the resulting pair is Morse--Smale (that is, the descending flow lines intersect transversely, cf. \cite{Abbondandolo2006}*{Section~2.12}, and the function is Palais--Smale). Hence, its Morse--Witten complex must compute the singular cohomology of $\cB$. However, by \Cref{lemma:Hessian_Maxwell}, we know that for any integer $N$, there is $\tau_N>0$ such that for all $\tau > \tau_N$ the index of $\cM$ is at least $N$ which then implies that the index of $\cE_\tau^\epsilon$ at any of the critical points in $\cM$ is at least $N$. Indeed, the index does not decrease under sufficiently small perturbations and so we can choose $\epsilon$ small enough so that the number of negative eigenvalues of $\Hess \left( \cE_\tau \right)$ and $\Hess \left( \cE_\tau^\epsilon \right)$ at the critical points of $\cE_\tau^\epsilon$ are the same. On the other hand, as $\cE$ attains its absolute minimum at $\cV$, the index of any of the critical points of $\cE_\tau^\epsilon$ in $\cV$ coincides with the index of $f_\cM$ which is at most $2d+1=\dim_\rl(\cV)$. Given that for any positive $\delta < \cE_\tau (\cM) - \cE_\tau (\cV)$ we have a retraction
	\begin{equation}
		\cE_\tau^{-1} (- \infty , \cE_\tau (\cV) + \delta] \cong \cV.
	\end{equation}
	Thus, the top degree cohomology class of $\cV$ induces a nonzero class in the degree $2d+1$ Morse--Witten cohomology of $\cE_\tau^{-1} (- \infty , \cE_\tau (\cV) + \delta]$. Thus, there is a closed $c_{2d+1} \in C^*_{MW} \left( \cE_\tau \right)$ of the Morse--Witten complex of $\cE_\tau^{-1}(- \infty , \cE_\tau (\cV) + \delta]$ which does not vanish in cohomology and so defines a nontrivial class
	\begin{equation}
		[c_{2d+1}] \in H^{2d+1}_{MW}(\cE_\tau, \cE_\tau^{-1} (- \infty , \cE_\tau (\cV)+ \delta] ).
	\end{equation}
	However, by \Cref{corollary:B_homotopy}, $\cB \cong_\mathbb{Q} K (H^1(\Sigma, \mathbb{Z}),1)$, which has trivial cohomology in degrees above $2 \: \mathrm{genus} \left( \Sigma \right)$. Hence, if $d \geqslant \mathrm{genus} \left( \Sigma \right)$, the class $[c_{2d+1}]$ must vanish in the Morse--Witten cohomology of $\cB$, that is
	\begin{equation}
		[c_{2d+1}] = 0 \in H^{2d+1}_{MW}(\cE_\tau,\cB).
	\end{equation}
	Hence, there must exist $c_{2d+2} \in C^{2d+2}(\cE_\tau, \cB)$ such that $\del c_{2d+2}=c_{2d+1}$ which is impossible if $N>2d+2$. This contradicts the hypothesis that there are no other critical points of $\cE_\tau$ other than those in $\cM \cup \cV$. Iterating this procedure we deduce the existence of at least $k=N-(2d+2)$ other critical points of $\cB$. 
\end{proof}

\bigskip

\section{Instability of nonminimal solutions}
\label{sec:instability}

In this section we prove \Cref{Mtheorem:instability}.

\smallskip

Let $d \eqdef c_1 (\cL) [\Sigma] \in \Z$ which (without any loss of generality) we assume to be nonnegative. Let $\Lambda$ be the contraction of 2-forms with $\vol_g$. Recall from \cite{B90}*{Corollary~2.3} that if $\tau_{\mathrm{Bradlow}} = \tfrac{4 \pi d}{\Ar \left( \Sigma, g \right)}$ and $\tau > \tau_{\mathrm{Bradlow}}$, then the absolute minimizers of the Ginzburg--Landau energy \eqref{eq:gle}, called vortex fields, are characterized by the \emph{vortex equations}
\begin{subequations}
	\begin{align}
		i \Lambda F_\nabla	&= \tfrac{1}{2} \left( \tau - |\upphi|^2 \right), \label{eq:vortex_1} \\
		\delbar_\nabla \upphi	&= 0. \label{eq:vortex_2}
	\end{align}
\end{subequations}
In \cite{N17b}*{Main Theorem~2} the first author showed that when $\tau \leqslant \tau_{\mathrm{Bradlow}}$, then the only critical points of the Ginzburg--Landau energy \eqref{eq:gle} are the normal phase solutions, which in this case are also absolute minimizers.

Absolute minimizers are necessarily \emph{stable}. We now show that when $\tau > 0$, then other critical points (for example, the ones given in \Cref{Mtheorem:nonminimal}) are necessarily \emph{unstable}. Recall that we call a critical point \emph{irreducible}, if $\upphi$ is not (identically) zero.

\smallskip

\begin{proof}[Proof of \Cref{Mtheorem:instability}:]
	The first claim was proved in \cite{N17b}*{Main Theorem~2}.

	Let us assume that $\tau > \tau_{\mathrm{Bradlow}}$. By \cite{B90}*{Proposition~2.1}, we have the following ``Bogomolny trick'' for \emph{all} $\left( \nabla, \upphi \right)$:
	\begin{align}
		\cE_\tau \left( \nabla, \upphi \right)	&= \int\limits_\Sigma \left( 2 |\delbar_\nabla \upphi|^2 + \left| i \Lambda F_\nabla - \tfrac{1}{2} \left( \tau - |\upphi|^2 \right) \right|^2 \right) \vol_g + 2 \pi \tau d. \label{eq:Bogomolny_trick}
	\end{align}
	This equality proves that solutions of the vortex \cref{eq:vortex_1,eq:vortex_2} are, in fact, absolute minimizers of the Ginzburg--Landau energy \eqref{eq:gle}.
	
	Let $\left( \nabla, \upphi \right)$ now be an irreducible critical point that is not a solution to the vortex \cref{eq:vortex_1,eq:vortex_2}. We now construct energy-decreasing directions for $\left( \nabla, \upphi \right)$. In order to do that let us investigate the following linear, elliptic PDE for $(a, \uppsi) \in i \Omega^1 \times \Omega_\cL^0$:
	\begin{subequations}
		\begin{align}
			\left( - i \Lambda \rd + \rd^* \right) a - h \left( \upphi, \uppsi \right)	&= 0, \label{eq:fake-tangent_1} \\
			\sqrt{2} \left( \delbar_\nabla \uppsi + a^{0,1} \upphi \right)			&= 0. \label{eq:fake-tangent_2}
		\end{align}
	\end{subequations}
	The factor of $\sqrt{2}$ is needed for later purposes. Let the space of solutions of \cref{eq:fake-tangent_1,eq:fake-tangent_2} be $\cT_{\left( \nabla, \upphi \right)}$. We remark that when $\left( \nabla, \upphi \right)$ is a vortex field, then \cref{eq:fake-tangent_1,eq:fake-tangent_2} are exactly the linearizations of the vortex \cref{eq:vortex_1,eq:vortex_2} with the Coulomb-type gauge fixing condition that $(a, \uppsi)$ is $L^2$-orthogonal to the gauge orbit through $\left( \nabla, \upphi \right)$.

	We show three things to complete the proof:
	\begin{enumerate}

		\itemsep0em
		\parskip0em
		
		\item $\cT_{\left( \nabla, \upphi \right)}$ has (real) dimension at least $2 d + 2$. In particular, $\cT_{\left( \nabla, \upphi \right)}$ is nontrivial.

		\item $\cT_{\left( \nabla, \upphi \right)}$ can be equipped with the structure of a complex vector space.
		
		\item Each complex line in $\cT_{\left( \nabla, \upphi \right)}$ has a real line that is a (strictly) energy-decreasing direction, meaning that (for $t$ small enough):
		\begin{equation}
			\cE_\tau \left( \nabla + t a, \upphi + t \uppsi \right) < \cE_\tau \left( \nabla, \upphi \right). \label[ineq]{ineq:instability}
		\end{equation}		
	\end{enumerate}
	
	Let us write \cref{eq:fake-tangent_1,eq:fake-tangent_2} as a single equation of the form
	\begin{equation}
		\bL_{\left( \nabla, \upphi \right)} (a, \uppsi) = 0,
	\end{equation}
	where $\bL_{\left( \nabla, \upphi \right)}$ is a first order and elliptic differential operator. In particular, $\bL_{\left( \nabla, \upphi \right)}$ is Fredholm, and we claim that its (real) Fredholm index is exactly $2 d$. In order to prove this claim, note that $\bL_{\left( \nabla, \upphi \right)}$ is a compact perturbation of the Fredholm operator $\bL_{(\nabla, 0)} = \left( - i \Lambda \rd + \rd^*, \sqrt{2} \delbar_\nabla \right)$. As $\bL_{(\nabla, 0)}$ is a direct sum of two Fredholm operators, its Fredholm index is the sum of the Fredholm indices of the operators
	\begin{align}
		L_1 &\eqdef - i \Lambda \rd + \rd^* : i \Omega^1 \rightarrow L^2 \left( X; \cx \right), \\
		L_2 &\eqdef \delbar_\nabla : \Omega_\cL^0 \rightarrow L^2 \left( X; \cL \otimes K_\Sigma^{- 1} \right).
	\end{align}
	It is easy to see that the kernel of $L_1$ consists of harmonic 1-form, while its cokernel consists of constant, complex functions only, thus
	\begin{equation}
		\index_\rl \left( L_1 \right) = 2 \: \index_\cx \left( L_1 \right) = 2 \left( \mathrm{genus} \left( \Sigma \right) - 1 \right).
	\end{equation}
	Finally, the kernel of $L_2$ consists of holomorphic sections of $\cL$, while its cokernel consists of anti-holomorphic section of $\cL \otimes K_\Sigma^{- 1}$, thus has complex dimension $h^0 \left( \cL^{- 1} \otimes K_\Sigma \right)$. Hence the (complex) index of $L_2$ is $h^0 (\cL) - h^0 \left( \cL^{- 1} K_\Sigma \right)$, and thus, by the Riemann--Roch Theorem, we get
	\begin{equation}
		\index_\rl \left( L_2 \right)	= 2 \left( h^0 (\cL) - h^0 \left( \cL^{- 1} \otimes K_\Sigma \right) \right) = 2 \left( d + 1 - \mathrm{genus} \left( \Sigma \right) \right).
	\end{equation}
	Thus
	\begin{align}
		\index_\rl \left( \bL_{\left( \nabla, \upphi \right)} \right)	&= \index_\rl \left( \bL_{(\nabla, 0)} \right) \\
															&= \index_\rl \left( L_1 \right) + \index_\rl \left( L_2 \right) \\
															&= 2 \left( \mathrm{genus} \left( \Sigma \right) - 1 \right) + 2 \left( d + 1 - \mathrm{genus} \left( \Sigma \right) \right) \\
															&= 2 d.
	\end{align}
	Hence $\cT_{\left( \nabla, \upphi \right)}$ is necessarily nontrivial if $d > 0$. Moreover, the Ginzburg--Landau \cref{eq:gl1,eq:gl2} imply that the pair $(f, \chi) = \left( \tfrac{\tau}{2} + \tfrac{1}{2} |\upphi|^2 - i \Lambda F_\nabla, \sqrt{2} \delbar_\nabla \upphi \right)$ is in the cokernel of $\bL_{\left( \nabla, \upphi \right)}$ (this is where the factor of $\sqrt{2}$ \cref{eq:fake-tangent_2} was needed). Furthermore, this pair is not zero, since $\left( \nabla, \upphi \right)$ is not a vortex field, so the cokernel of $\bL_{\left( \nabla, \upphi \right)}$ is also nontrivial. Hence $\cT_{\left( \nabla, \upphi \right)}$ has (real) dimension at least $2 d + 1 > 0$.
	
	Next we define the map $I$ acting on $(a, \uppsi) \in \cT_{\left( \nabla, \upphi \right)}$ as
	\begin{equation}
		I (a, \uppsi) \eqdef (\ast a, i \uppsi). \label{eq:cx_str}
	\end{equation}
	Using that $\ast^2 a = - a$, $(\ast a)^{0, 1} = i a^{0, 1}$, $\rd^* a = - \ast \rd \ast a$, $\ast \rd a = \ast \rd a$, and $h (i \uppsi, \upphi) = - i h (\uppsi, \upphi)$, we see that $I$ preserves $\cT_{\left( \nabla, \upphi \right)}$, and $I^2 = - \id_{\cT_{\left( \nabla, \upphi \right)}}$. Thus $I$ is a complex structure on $\cT_{\left( \nabla, \upphi \right)}$. Since $\cT_{\left( \nabla, \upphi \right)}$ is complex and has real dimension at least $2 d + 1$, it, in fact, has real dimension at least $2 d + 2$. This proves the first two bullet points.

	For any $(a, \uppsi) \in i \Omega^1 \times \Omega_\cL^0$ let us inspect the difference
	\begin{equation}
		\delta \cE (t) \eqdef \cE_\tau \left( \nabla + t a, \upphi + t \uppsi \right) - \cE_\tau \left( \nabla, \upphi \right).
	\end{equation}
	The $O (t)$ term vanishes, because $\left( \nabla, \upphi \right)$ is a critical point. The $O (t^2)$ term can be computed using \cref{eq:Bogomolny_trick} as
	\begin{equation}
	\begin{aligned}
		\lim\limits_{t \rightarrow 0} \frac{\delta \cE (t)}{t^2} &= \int\limits_\Sigma \left( 2 \left| \delbar_\nabla \uppsi + a^{0,1} \upphi \right|^2 + 4 \Re \left( \langle \delbar_\nabla \upphi | a^{0,1} \uppsi \rangle \right) \right) \ \vol_g \\
			& \quad + \int\limits_\Sigma \left( \left( i \Lambda \rd a + \Re \left( h (\uppsi, \upphi) \right) \right)^2 + \left( i \Lambda F_\nabla - \tfrac{1}{2} \left( \tau - |\upphi|^2 \right) \right) |\uppsi|^2 \right) \ \vol_g. \label{eq:hessian}
	\end{aligned}
	\end{equation}
	
	Let us assume now that $(a, \uppsi) \in \cT_{\left( \nabla, \upphi \right)} - \{ (0, 0) \}$. Then we get
	\begin{equation}
		\lim\limits_{t \rightarrow 0} \frac{\delta \cE (t)}{t^2} = 4 \underbrace{\int\limits_\Sigma \Re \left( \langle \delbar_\nabla \upphi | a^{0,1} \uppsi \rangle \right) \ \vol_g}_{\cI_1} + \underbrace{\int\limits_\Sigma \left( i \Lambda F_\nabla - \tfrac{1}{2} \left( \tau - |\upphi|^2 \right) \right) |\uppsi|^2 \ \vol_g}_{\cI_2}.
	\end{equation}
	The first term, $\cI_1$, is not invariant under the action of $\rU (1)$ on $\cT_{\left( \nabla, \upphi \right)}$, instead if $\mu \in \rU (1)$, then 
	\begin{equation}
		\Re \left( \langle \delbar_\nabla \upphi | (\mu a)^{0, 1} (\mu \uppsi) \rangle \right) = \Re \left( \mu^2 \langle \delbar_\nabla \upphi | a^{0,1} \uppsi \rangle \right).
	\end{equation}
	Thus if $\int_\Sigma \langle \delbar_\nabla \upphi | a^{0,1} \uppsi \rangle \vol_g = r \exp (i \theta)$, with $r \geqslant 0$, then let $\mu = \exp \left(i (\pi - \theta)/2 \right)$, and change $(a, \uppsi)$ to $\mu (a, \uppsi)$. With this $\cI_1 = - r \leqslant 0$.
	
	As opposed to $\cI_1$, the second term, $\cI_2$, is invariant under the action of $\rU (1)$, and since $\left( \nabla, \upphi \right)$ is irreducible, but not a vortex field, we have by \cite{N17b}*{Lemma~4.1} that
	\begin{equation}
		i \Lambda F_\nabla - \tfrac{1}{2} \left( \tau - |\upphi|^2 \right) < 0,
	\end{equation}
	holds everywhere on $\Sigma$. Finally, note that since both $\upphi$ and $(a, \uppsi)$ are both smooth and nonzero, then, using \cref{eq:fake-tangent_1,eq:fake-tangent_2}, we can show that $\uppsi$ is also nonzero. Thus $\cI_2$ is strictly negative, which completes the proof of \cref{ineq:instability}, and hence of the theorem.
\end{proof}

We learned the proof of the last theorem of this section from Da Rong Cheng, who in turn claims that the key trick in the proof is rather well-known among experts of minimal submanifolds. In any case, we claim no ownership of the following result, but present it for the sake of completeness.

\begin{theorem}
\label{theorem:index_bound}
	Under the hypotheses above, let $\left( \nabla, \upphi \right)$ be an irreducible critical point that is not a vortex field. Then the (real) Morse-index of the Ginzburg--Landau energy \eqref{eq:gle} at $\left( \nabla, \upphi \right)$ is at least $d + 1$.
\end{theorem}

\begin{proof}
	Let $\cH$ be the $L^2$-completion of the real pre-Hilbert space of pairs of imaginary-valued 1-forms and section of $\cL$, and let $\mathbb{H}$ be the densely defined, self-adjoint operator that is the metric dual of the Hessian in \cref{eq:hessian}. By elliptic regularity, $\mathbb{H}$ has an orthonormal eigenbasis, $\{ x_1, x_2, \ldots \}$. From \cref{eq:hessian} it also follows that the spectrum of $\mathbb{H}$ is bounded from below. Let us label the eigenvectors, so that the corresponding eigenvalues satisfy that $\lambda_1 \leqslant \lambda_2 \ldots$, and let
	\begin{equation}
		E \eqdef \mathrm{span} \left( x_1, I (x_1), x_2, I (x_2), \ldots, x_d, I (x_d) \right) \cong \rl^{2 d + 2}.
	\end{equation}
	Then we get that
	\begin{align}
		\lambda_{d + 1}	&= \inf \left( \left\{ \ \frac{\langle x | \mathbb{H} (x) \rangle_\cH}{\| x \|_\cH^2} \ \middle| \ x \in \left( \mathrm{span} \left( x_1, x_2, \ldots, x_d \right) \right)^\perp - \{ 0 \} \ \right\} \right) \\
			&\leqslant \inf \left( \left\{ \ \frac{\langle x | \mathbb{H} (x) \rangle_\cH}{\| x \|_\cH^2} \ \middle| \ x \in E^\perp - \{ 0 \} \ \right\} \right).
	\end{align}
	Since $\cT_{\left( \nabla, \upphi \right)}$ has real dimension at least $2 d + 2$. Thus $E^\perp \cap \cT_{\left( \nabla, \upphi \right)}$ cannot be trivial. Let $x \in E^\perp \cap \cT_{\left( \nabla, \upphi \right)}$ have unit norm. Since both $E$ and $\cT_{\left( \nabla, \upphi \right)}$ are invariant under the action of $I$, which is a unitary transformation, we have that $I (x)$ is also in $E^\perp \cap \cT_{\left( \nabla, \upphi \right)}$. Thus, as in the proof of \Cref{Mtheorem:instability}, there is unit length complex number $a + b i$, thus that, if we replace $x$ with $\left( a + b I \right) x$, then $\langle x | \mathbb{H} (x) \rangle_\cH < 0$, and hence $\lambda_{d + 1} < 0$. Thus $\mathbb{H}$ has at least $d + 1$ negative eigenvalues. This concludes the proof.
\end{proof}

\bigskip

\appendix
\section{$\cE_\tau$ is Morse--Bott near $\cV \subset \cB$ when $\tau > \tau_{\mathrm{Bradlow}}$}
\label{app:MB_of_V}

We first show that $\cV$ is a smooth submanifold of $\cB$.

\smallskip

Recall that $\cX = \cC_{\cL, \rd^*} \times \Omega_\cL^0$, $H^1 \left( \Sigma; \Z \right)$ acts on $\cX$ via harmonic gauge transformations that are the identity at a fixed point, and $\cB = \cX / H^1 \left( \Sigma; \Z \right)$. Then we define the space
\begin{equation}
	\widetilde{\cV} \eqdef \{ \ \left( \nabla, \upphi \right) \in \cX \ | \ \left( \nabla, \upphi \right) \mbox{ solves the $\tau$-vortex \cref{eq:vortex_1,eq:vortex_2}} \ \}, \label{eq:V_tilde_def}
\end{equation}
which has the property that $\cV = \widetilde{\cV} / H^1 \left( \Sigma; \Z \right)$. Assuming that $\tau$ is above the Bradlow limit, every element of $\widetilde{\cV}$ is irreducible, hence if $\widetilde{\cV}$ is a smooth manifold, then so is $\cV$. 

\smallskip

First, we prove the following:

\begin{lemma}
	$\widetilde{\cV}$ is a smooth submanifold of $\cX$.
\end{lemma}

\begin{proof}
	We can view $\cX$ as an affine, real Hilbert manifold, and thus it is enough to show that $\widetilde{\cV}$ is a zero locus of a Fredholm map for which zero is a regular value. Let us define a smooth map via
	\begin{equation}
		\nu : \cX \rightarrow L^2 (\Sigma; \rl) \times \Omega_\cL^{0, 1}; \: \left( \nabla, \upphi \right) \mapsto \left( \tfrac{1}{2} \left( \tau - |\upphi|^2 \right) - i \Lambda F_\nabla, \sqrt{2} \: \delbar_\nabla \upphi \right).
	\end{equation}
	Clearly, $\widetilde{\cV}$ is exactly the zero locus of $\nu$. Let us identify $T \widetilde{\cV}$ with $\Omega^{0, 1} \times \Omega_\cL^0$, using the identifications of $i \Omega^1$ with $\Omega^{0, 1}$ via $a \mapsto \alpha \eqdef \sqrt{2} a^{0, 1}$, which is a unitary isomorphism. Then the derivative of $\nu$ has the form
	\begin{equation}
		\left( D \nu \right)_{\left( \nabla, \upphi \right)} \left( \alpha, \uppsi \right) = \left( \Re \left( \sqrt{2} \: \delbar^* \alpha - h \left( \upphi, \uppsi \right) \right), \sqrt{2} \: \delbar_\nabla \uppsi + \alpha \upphi \right). \label{eq:D_nu}
	\end{equation}
	The Reader can find details of the computation of \cref{eq:D_nu} in \cite{N17a}*{Lemma~1.2~and~1.3}. Note that $\left( D \nu \right)_{\left( \nabla, \upphi \right)}$ is Fredholm of index
	\begin{align}
		\index_\rl \left( \left( D \nu \right)_{\left( \nabla, \upphi \right)} \right) = \index_\rl \left( \Re \circ \delbar^* \right) + \index_\rl \left( \delbar_\nabla \right) = \index_\rl \left( \Re \circ \delbar^* \right) + 2 \: \index_\cx \left( \delbar_\nabla \right)
	\end{align}
	For any $\alpha \in \Omega^{0, 1}$, $\Re \left( \delbar^* \alpha \right) = 0$ exactly if $\alpha$ is anti-holomorphic, thus the kernel of $\Re \circ \delbar^*$ has real dimension $2 \: \mathrm{genus} \left( \Sigma \right)$. Its adjoint is $\delbar$ on $L^2 (\Sigma; \rl)$, thus the cokernel of $\Re \circ \delbar^*$ consists of (real) constants only. In the above formula $\delbar_\nabla$ is the Cauchy--Riemann operator from $\Omega_\cL^{0,1}$ to $\Omega_\cL^0$, thus by the Riemann--Roch Theorem, we get
	\begin{equation}
		\index_\cx \left( \delbar_\nabla \right) = d + 1 - \mathrm{genus} \left( \Sigma \right).
	\end{equation}
	Thus
	\begin{align}
		\index_\rl \left( \left( D \nu \right)_{\left( \nabla, \upphi \right)} \right) &= \index_\rl \left( \Re \circ \delbar^* \right) + 2 \: \index_\cx \left( \delbar_\nabla \right) \\
			&= \left( 2 \: \mathrm{genus} \left( \Sigma \right) - 1 \right) + 2 \left( d + 1 - \mathrm{genus} \left( \Sigma \right) \right) \\
			&= 2 d + 1.
	\end{align}
	The adjoint of $\left( D \nu \right)_{\left( \nabla, \upphi \right)}$ is
	\begin{equation}
		\left( D \nu \right)_{\left( \nabla, \upphi \right)}^* \left( f, \upxi \right) = \left( \sqrt{2} \: \delbar f + h \left( \upphi, \upxi \right), \sqrt{2} \: \delbar_\nabla^* \upxi - f \upphi \right). \label{eq:dual_linearization}
	\end{equation}
	The same operator as in \cref{eq:dual_linearization} was studied in \cite{N17a}*{Lemma~1.4~and~Corollary~1.5}, and thus we get that the kernel of $\left( D \nu \right)_{\left( \nabla, \upphi \right)}^*$ is trivial, which concludes the proof that $\widetilde{\cV}$ (and thus $\cV$) is a smooth manifold of dimension $2 d + 1$. The compactness of $\cV$ follows from the gauged Palais--Smale property of the Ginzburg--Landau energy \eqref{eq:gle}; cf. \cite{N17b}*{Lemma~3.1}.

	This proves that $\cV$ is a smooth submanifold of $\cB$.
\end{proof}

\medskip

Since $\cE_\tau$ is gauge invariant, we use the same notation for all of its descendants as well. We prove the following, which implies that $\cE_\tau$ is Morse--Bott near $\cV \subset \cB$ when $\tau > \tau_{\mathrm{Bradlow}}$.

\begin{lemma}
	Let $\left( \nabla, \upphi \right) \in \widetilde{\cV} \subset \cX$ be a solution to the $\tau$-vortex \cref{eq:vortex_1,eq:vortex_2}. Let
	\begin{equation}
		\cT_{\left( \nabla, \upphi \right)} \eqdef \left\{ \ \left( a, \uppsi \right) \in \Omega_{\rd^*}^1 \times \Omega_\cL^0 \ \middle| \ \left( \sqrt{2} a^{0, 1}, \uppsi \right) \in \ker \left( \left( D \nu \right)_{\left( \nabla, \upphi \right)} \right) \ \right\}.
	\end{equation}
	Then has a (real) dimension $2 d + 1$ and for all $\left( a, \uppsi \right) \in \cT_{\left( \nabla, \upphi \right)}$, we have
	\begin{equation}
		\Hess \left( \cE_\tau \right)_{\left( \nabla, \upphi \right)} \left( a, \uppsi \right) = 0. \label{eq:Hess_vanishing}
	\end{equation}
	Furthermore, there is a positive number, $\lambda$, such that, if a pair $\left( a, \uppsi \right) \in \Omega_{\rd^*}^1 \times \Omega_\cL^0$ is $L^2$-orthogonal to $\cT_{\left( \nabla, \upphi \right)}$, then
	\begin{equation}
		\Hess \left( \cE_\tau \right)_{\left( \nabla, \upphi \right)} \left( \left( a, \uppsi \right), \left( a, \uppsi \right) \right) \geqslant \lambda \| \left( a, \uppsi \right) \|_{\Omega_{\rd^*}^1 \times \Omega_\cL^0}^2. \label[ineq]{ineq:Hessian_bound}
	\end{equation}
\end{lemma}

\begin{proof}
	Recall from \Cref{app:MB_of_V} that the kernel of $\left( D \nu \right)_{\left( \nabla, \upphi \right)}$ has a real dimension $2 d + 1$. Since the map
	\begin{equation}
		\cT_{\left( \nabla, \upphi \right)} \rightarrow \ker \left( \left( D \nu \right)_{\left( \nabla, \upphi \right)} \right) : \: \left( a, \uppsi \right) \mapsto \left( \sqrt{2} a^{0, 1}, \uppsi \right),
	\end{equation}
	is norm-preserving and the target is finite dimensional, we get that $\cT_{\left( \nabla, \upphi \right)} \cong \ker \left( \left( D \nu \right)_{\left( \nabla, \upphi \right)} \right)$, as real vector spaces.

	Using the same computation that gave us \cref{eq:hessian} and combining it with the $\tau$-vortex \cref{eq:vortex_1,eq:vortex_2}, we get for any $\left( a, \uppsi \right) \in \Omega_{\rd^*}^1 \times \Omega_\cL^0$, we have
	\begin{equation}
		\Hess \left( \cE_\tau \right)_{\left( \nabla, \upphi \right)} \left( \left( a, \uppsi \right), \left( a, \uppsi \right) \right) = \int\limits_\Sigma \left( 2 \left| \delbar_\nabla \uppsi + a^{0,1} \upphi \right|^2 + \left( i \Lambda \rd a + \Re \left( h (\uppsi, \upphi) \right) \right)^2 \right) \ \vol_g. \label{eq:Hess_at_vortex}
	\end{equation}
	Using once again $\alpha \eqdef \sqrt{2} a^{0, 1}$, we can rewrite \cref{eq:Hess_at_vortex} as
	\begin{equation}
		\Hess \left( \cE_\tau \right)_{\left( \nabla, \upphi \right)} \left( \left( a, \uppsi \right), \left( a, \uppsi \right) \right) = \| \left( D \nu \right)_{\left( \nabla, \upphi \right)} \left( \alpha, \uppsi \right) \|_{L^2}^2. \label{eq:Hess_at_vortex_rewrite}
	\end{equation}
	Thus if $\left( a, \uppsi \right) \in \cT_{\left( \nabla, \upphi \right)}$, then we get \cref{eq:Hess_vanishing}.

	We prove \cref{ineq:Hessian_bound} by contradiction. Assume that \cref{ineq:Hessian_bound} does not hold and choose a sequence, $\left( a_n, \uppsi_n \right)_{n \in \N}$, such that for all $n \in N$
	\begin{align}
		\left( a_n, \uppsi_n \right) 																&\perp_{L^2} \cT_{\left( \nabla, \upphi \right)}, \\
		\| \left( a_n, \uppsi_n \right) \|_{\Omega_{\rd^*}^1 \times \Omega_\cL^0}					&= 1, \\
		\Hess \left( \cE_\tau \right)_{\left( \nabla, \upphi \right)} \left( a_n, \uppsi_n \right)	&= \frac{1}{n^2}.
	\end{align}
	Let $\alpha_n \eqdef \sqrt{2} a_n^{0, 1}$. Then
	\begin{align}
		\left( \alpha_n, \uppsi_n \right) 																				&\perp_{L^2} \ker \left( \left( D \nu \right)_{\left( \nabla, \upphi \right)} \right), \label[cond]{cond:perp} \\
		\| \left( \alpha_n, \uppsi_n \right) \|_{\Omega^{0, 1} \times \Omega_\cL^0}										&= 1, \label[cond]{cond:norm} \\
		\| \left( D \nu \right)_{\left( \nabla, \upphi \right)} \left( \alpha_n, \uppsi_n \right) \|_{L^2}	&= \frac{1}{n}. \label[cond]{cond:bound}
	\end{align}
	But $\left( D \nu \right)_{\left( \nabla, \upphi \right)}$ considered as an operator from $\Omega^{0, 1} \times \Omega_\cL^0$ onto the $L^2$-completion of the space of pairs of smooth $(0, 1)$-forms and sections of $\cL$, and thus by \cite{BS18}*{Theorem 4.1.16 (Closed Image Theorem), Ineq. (4.1.7)}, we get that there is a positive number $C$, such that, for all $n \in \N$
	\begin{equation}
	\begin{aligned}
		\inf \left( \left\{ \ \| \left( \alpha_n, \uppsi_n \right) - \left( \alpha, \uppsi \right) \|_{\Omega^{0, 1} \times \Omega_\cL^0} \ \middle| \ \left( \alpha, \uppsi \right) \in \ker \left( \left( D \nu \right)_{\left( \nabla, \upphi \right)} \right) \ \right\} \right) \leqslant \\
		C \| \left( D \nu \right)_{\left( \nabla, \upphi \right)} \left( \alpha_n, \uppsi_n \right) \|_{L^2}. \label[ineq]{ineq:closed_image}
	\end{aligned}
	\end{equation}
	Using \cref{cond:perp,cond:norm}, we get
	\begin{equation}
		\| \left( \alpha_n, \uppsi_n \right) - \left( \alpha, \uppsi \right) \|_{\Omega^{0, 1} \times \Omega_\cL^0}^2 = \| \left( \alpha_n, \uppsi_n \right) \|_{\Omega^{0, 1} \times \Omega_\cL^0}^2 + \| \left( \alpha, \uppsi \right) \|_{\Omega^{0, 1} \times \Omega_\cL^0}^2 = 1 + \| \left( \alpha, \uppsi \right) \|_{\Omega^{0, 1} \times \Omega_\cL^0}^2,
	\end{equation}
	and thus
	\begin{equation}
		\forall n \in \N : \quad \inf \left( \left\{ \ \| \left( \alpha_n, \uppsi_n \right) - \left( \alpha, \uppsi \right) \|_{\Omega^{0, 1} \times \Omega_\cL^0} \ \middle| \ \left( \alpha, \uppsi \right) \in \ker \left( \left( D \nu \right)_{\left( \nabla, \upphi \right)} \right) \ \right\} \right) = 1.
	\end{equation}
	Combining this with \cref{cond:bound,ineq:closed_image}, we get that for all $n \in \N$, $1 \leqslant \tfrac{C}{n}$, which is a contradiction.
\end{proof}

\smallskip

\begin{corollary}
\label{corollary:Morse--Bott}
	$\cE_\tau$ is Morse--Bott function, in the sense of \cite{Feehan2020}*{Definition~1.9}, near $\cV \subset \cB$ when $\tau > \tau_{\mathrm{Bradlow}}$.
\end{corollary}

\begin{proof}
	Recall that $\widetilde{\cV}$ is the $H^1 \left( \Sigma; \Z \right)$-cover of $\cV$ and $\cE_\tau$ is gauge invariant, so it is enough to work on $\widetilde{\cV}$. We have already proved that $\widetilde{\cV} \subset \cX$ is a smooth submanifold. For any $\left( \nabla, \upphi \right) \in \widetilde{\cV}$, the kernel of $D \cE_\tau$ at $\left( \nabla, \upphi \right)$ is $\cT_{\left( \nabla, \upphi \right)}$, which is finite dimensional and thus has a closed (orthogonal) complement. Considered as a map from $T_{\left( \nabla, \upphi \right)} \cX$ to its dual, the Hessian is a Fredholm operator with index zero (as it is the metric dual of a bilinear map). The only thing left to be proven from \cite{Feehan2020}*{Definition~1.9} is that the image of the Hessian is exactly the space of metric duals of vectors orthogonal to $\cT_{\left( \nabla, \upphi \right)}$. Since the kernel of the Hessian is also $\cT_{\left( \nabla, \upphi \right)}$, and the Hessian is Fredholm and symmetric, this is again immediate.
\end{proof}

	\bibliography{references}

\begin{bibdiv}
\begin{biblist}

\bib{Abbondandolo2006}{incollection}{
      author={Abbondandolo, Alberto},
      author={Majer, Pietro},
       title={{Lectures on the Morse complex for infinite-dimensional
  manifolds}},
        date={2006},
   booktitle={Morse theoretic methods in nonlinear analysis and in symplectic
  topology},
   publisher={Springer},
       pages={1\ndash 74},
}

\bib{B90}{article}{
      author={Bradlow, S.~B.},
       title={Vortices in holomorphic line bundles over closed {K}\"ahler
  manifolds},
        date={1990},
        ISSN={0010-3616},
     journal={Commun. Math. Phys.},
      volume={135},
      number={1},
       pages={1\ndash 17},
         url={https://projecteuclid.org/euclid.cmp/1104201917},
      review={\MR{1086749 (92f:32053)}},
}

\bib{BS18}{book}{
      author={B\"{u}hler, Theo},
      author={Salamon, Dietmar~A.},
       title={Functional analysis},
      series={Graduate Studies in Mathematics},
   publisher={American Mathematical Society, Providence, RI},
        date={2018},
      volume={191},
        ISBN={978-1-4704-4190-6},
         url={https://doi-org.proxy.library.ucsb.edu:9443/10.1090/gsm/191},
      review={\MR{3823238}},
}

\bib{C20}{article}{
      author={Cheng, D-R.},
       title={{Stable solutions to the abelian Yang--Mills--Higgs equations on
  $S^2$ and $T^2$}},
        date={2021},
     journal={J. Geom. Anal.},
         url={https://doi.org/10.1007/s12220-021-00619-y},
}

\bib{CERS17}{article}{
      author={Chouchkov, D.},
      author={Ercolani, N.~M.},
      author={Rayan, S.},
      author={Sigal, I.~M.},
       title={{Ginzburg--Landau equations on Riemann surfaces of higher
  genus}},
        date={2020},
        ISSN={0294-1449},
     journal={Ann. Inst. H. Poincar\'{e} Anal. Non Lin\'{e}aire},
      volume={37},
      number={1},
       pages={79\ndash 103},
  url={https://doi-org.proxy.library.ucsb.edu:9443/10.1016/j.anihpc.2019.04.002},
      review={\MR{4049917}},
}

\bib{Donaldson1990}{book}{
      author={Donaldson, S.~K.},
      author={Kronheimer, P.~B.},
       title={The geometry of four-manifolds},
      series={Oxford Mathematical Monographs},
   publisher={The Clarendon Press Oxford University Press},
     address={New York},
        date={1990},
        ISBN={0-19-853553-8},
        note={Oxford Science Publications},
      review={\MR{MR1079726 (92a:57036)}},
}

\bib{Feehan2020}{article}{
      author={Feehan, Paul~MN},
       title={{On the Morse--Bott property of analytic functions on Banach
  spaces with {\L}ojasiewicz exponent one half}},
        date={2020},
     journal={Calculus of Variations and Partial Differential Equations},
      volume={59},
      number={2},
       pages={1\ndash 50},
}

\bib{GP93}{article}{
      author={Garc\'ia-Prada, O.},
       title={Invariant connections and vortices},
        date={1993},
     journal={Comm. Math. Phys.},
      volume={156},
      number={3},
       pages={527\ndash 546},
         url={https://projecteuclid.org/euclid.cmp/1104253716},
}

\bib{JT80}{book}{
      author={Jaffe, A.},
      author={Taubes, C.~H.},
       title={{Vortices and Monopoles}},
      series={Progress in Physics},
   publisher={Birkh\"auser, Boston, MA},
        date={1980},
        ISBN={3-7643-3025-2},
      review={\MR{614447 (82m:81051)}},
}

\bib{N17a}{article}{
      author={Nagy, \'{A}.},
       title={{The Berry connection of the Ginzburg--Landau vortices}},
        date={2017},
        ISSN={0010-3616},
     journal={Comm. Math. Phys.},
      volume={350},
      number={1},
       pages={105\ndash 128},
         url={https://doi.org/10.1007/s00220-016-2701-0},
      review={\MR{3606471}},
}

\bib{N17b}{article}{
      author={Nagy, \'{A}.},
       title={{Irreducible Ginzburg--Landau fields in dimension 2}},
        date={2018},
        ISSN={1050-6926},
     journal={J. Geom. Anal.},
      volume={28},
      number={2},
       pages={1853\ndash 1868},
         url={https://doi.org/10.1007/s12220-017-9890-4},
      review={\MR{3790522}},
}

\bib{nagy_bifurcation_2022}{article}{
      author={Nagy, \'{A}.},
      author={Oliveira, G.},
       title={{On the bifurcation theory of the Ginzburg--Landau energy
  functional}},
        date={2022},
}

\bib{Parker1993}{article}{
      author={Parker, T.~H.},
       title={{E}quivariant sobolev theorems and {Y}ang--{M}ills--{H}iggs
  fields},
        date={1993},
     journal={Global Analysis in Modern Mathematics, Publish or Perish,
  Houston},
}

\bib{Parker1992b}{article}{
      author={Parker, Thomas~H.},
       title={A {M}orse theory for equivariant {Y}ang--{M}ills},
        date={1992},
        ISSN={0012-7094},
     journal={Duke Math. J.},
      volume={66},
      number={2},
       pages={337\ndash 356},
  url={https://doi-org.proxy.library.ucsb.edu:9443/10.1215/S0012-7094-92-06610-5},
      review={\MR{1162193}},
}

\bib{Parker1992a}{article}{
      author={Parker, Thomas~H.},
       title={Nonminimal {Y}ang--{M}ills fields and dynamics},
        date={1992},
        ISSN={0020-9910},
     journal={Invent. Math.},
      volume={107},
      number={2},
       pages={397\ndash 420},
         url={https://doi-org.proxy.library.ucsb.edu:9443/10.1007/BF01231895},
      review={\MR{1144429}},
}

\bib{PS20}{article}{
      author={Pigati, A.},
      author={D., Stern},
       title={{Minimal submanifolds from the abelian Higgs model}},
        date={2020},
     journal={Invent. Math.},
         url={https://doi.org/10.1007/s00222-020-01000-6},
}

\end{biblist}
\end{bibdiv}

\end{document}